\documentclass[a4paper,12pt,reqno]{amsart}

\usepackage[
  a4paper,
  top=2.4cm,
  bottom=2.4cm,
  left=2.2cm,
  right=2.2cm
]{geometry}
\usepackage[british]{babel}
\usepackage[T1]{fontenc}
\usepackage{lmodern}
\usepackage[utf8]{inputenc}
\usepackage{amsmath,amssymb,amsthm,mathtools}
\usepackage[hidelinks]{hyperref}
\usepackage[capitalise,noabbrev]{cleveref}
\usepackage{setspace}
\usepackage{xcolor} 
\hypersetup{
  pdftitle={Monochromatic cycle partitions in 3-mean edge-colourings},
  pdfauthor={Richard Lang and Guilherme Oliveira Mota},
  pdfsubject={Monochromatic cycle partitions under a mean colour-degree bound},
  pdfkeywords={mean edge-colouring, local edge-colouring, monochromatic cycle,
    cycle partition}
}

\newtheorem{theorem}{Theorem}[section]
\newtheorem{lemma}[theorem]{Lemma}

\newtheorem{claim}[theorem]{Claim}

\crefname{theorem}{Theorem}{Theorems}
\crefname{lemma}{Lemma}{Lemmas}
\crefname{remark}{Remark}{Remarks}
\crefname{observation}{Observation}{Observations}

\newcommand{\cP}{\mathcal P}
\newcommand{\dd}{d_{\chi}}

\title{Monochromatic cycle partitions in $3$-mean edge-colourings}
\author[C.~Bir\'o]{Csaba Bir\'o}
\address{Department of Mathematics, University of Louisville, Louisville, KY, USA}
\email{csaba.biro@louisville.edu}
\author[R.~Lang]{Richard Lang}
\address{Departament de Matem\`atiques, Universitat Polit\`ecnica de Catalunya,
  Barcelona, Spain; and Centre de Recerca Matem\`atica, Barcelona, Spain}
\email{richard.lang@upc.edu}
\author[R.~R.~Martin]{Ryan R. Martin}
\address{Department of Mathematics, Iowa State University, Ames, IA, USA}
\email{rymartin@iastate.edu}
\author[G.~O.~Mota]{Guilherme Oliveira Mota}
\address{Departamento de Ci\^encia da Computa\c{c}\~ao, Instituto de Matem\'atica,
  Estat\'istica e Ci\^encia da Computa\c{c}\~ao, Universidade de S\~ao Paulo,
  Rua do Mat\~ao 1010, 05508-090 S\~ao Paulo, Brazil}
\email{mota@ime.usp.br}
\author[A.~Pokrovskiy]{Alexey Pokrovskiy}\address{Department of Mathemtatics, University College London, UK}\email{dralexeypokrovskiy@gmail.com}
\author[M.~Ruszink\'o]{Mikl\'os Ruszink\'o}
\address{Sorbonne University Abu Dhabi, SAFIR, CEFAMSDS, Abu Dhabi, UAE; and Alfr\'ed R\'enyi Institute of Mathematics, Budapest, Hungary}
\email{ruszinko@renyi.hu}
\author[G.~S\'ark\"ozy]{G\'abor N. S\'ark\"ozy}
\address{Computer Science Department, Worcester Polytechnic Institute, Worcester, MA, USA; and Alfr\'ed R\'enyi Institute of Mathematics, Budapest, Hungary}
\email{gsarkozy@wpi.edu}

\date{}
\begin{document}

\begin{abstract}
  Given $r \in \mathbb{N}$ and an edge-coloured complete graph such that the average number of colours incident with a vertex is at most $r$, Conlon and Stein asked whether there is a vertex partition into a bounded
  number of monochromatic cycles, and proved this for $r=2$. We settle the
  first open case by proving the corresponding statement for $r=3$.
\end{abstract}

\maketitle

\section{Introduction}
Problems on monochromatic coverings and partitions in edge-colourings of
graphs can be traced back to a conjecture posed by Lehel in the 1970s: every
red/blue edge-colouring of the complete graph $K_n$ contains a red cycle and a
blue cycle whose vertex sets form a partition of $V(K_n)$, where edges,
vertices, and the empty set are regarded as cycles. 
Following earlier work of
Gy\'arf\'as~\cite{Gyarfas1983}, {\L}uczak, R\"odl and
Szemer\'edi~\cite{LuczakRodlSzemeredi1998} proved this conjecture for
sufficiently large~$n$.
Subsequently, Bessy and Thomass\'e~\cite{BessyThomasse2010} established the conjecture for every~$n$.

Since the 1990s, research on monochromatic cycle partitions has developed in
several directions. A prominent question asks how much of Lehel's conjecture
remains valid when more than two colours are allowed. Erd\H{o}s, Gy\'arf\'as and
Pyber~\cite{ErdosGyarfasPyber1991} proved that the vertex set of every
$r$-edge-coloured complete graph can be partitioned into at most $25r^2\log r$
monochromatic cycles, and conjectured that $r$ cycles should always be enough.
For sufficiently large~$n$, Gy\'arf\'as, Ruszink\'o, S\'ark\"ozy and
Szemer\'edi~\cite{GRSS2006} improved the upper bound to $100r\log r$. The proposed
bound of $r$ cycles, however, was disproved by Pokrovskiy~\cite{Pokrovskiy2014},
who constructed, for every $r\geq 3$ and infinitely many values of $n$, an
$r$-edge-colouring of $K_n$ for which at least $r+1$ cycles are necessary.
Most recently, Li and Wang~\cite{li2026superlinear} improved this to $\Omega(r \log\log r)$ necessary cycles.
Determining the optimal bound remains a hard problem, even for a small
number of colours. Related host graph variants replace the complete
graph by graphs satisfying minimum-degree or independence-number conditions
\cite{AllenEtAl2024,DeBiasioNelsen2017,DiBraccioPatel2026,
KorandiLangLetzterPokrovskiy2021,Letzter2019,Sarkozy2011}. Random graph
analogues were initiated by Kor\'andi, Mousset, Nenadov, \v{S}kori\'c and
Sudakov~\cite{KorandiEtAl2018} and further developed by Lang and
Lo~\cite{LangLo2021}.

Variants of covering and partitioning problems that impose conditions 
weaker than a bound on the total number of colours have also attracted
considerable attention; see the survey of Gy\'arf\'as~\cite{Gyarfas2016}. Given
an edge-colouring $\chi$ of a graph $G$, define the \emph{palette} of $v$ as
  \[
  \cP(v):=\{\chi(vu)\colon vu\in E(G)\}\,,
  \]
and let $\dd(v):=|\cP(v)|$ be the \emph{colour-degree} of $v$. We could allow an
unbounded number of colours, while restricting the number of colours on edges
incident to each vertex as in the following definition: a colouring $\chi$ is
\emph{$r$-local} if $\dd(v)\le r$ for every vertex $v$. Even though the number
of colours can be unbounded, the local hypothesis of an $r$-local colouring
provides a uniform bound for the colour-degree at every vertex.

Conlon and Stein \cite{ConlonStein2016} proved that every $r$-locally coloured
complete graph has a partition into $O(r^2\log r)$ monochromatic cycles; for
$r=2$, they showed that two cycles of different colours suffice. For
sufficiently large graphs, Lang and Stein~\cite{LangStein2017}
obtained the explicit bounds $2r^2$ for $r$-locally coloured complete graphs
and $4r^2$ for $r$-locally coloured balanced complete bipartite graphs.
Complete bipartite host graphs have also been studied under global colour
bounds~\cite{BenevidesQuintinoTalon2025,Haxell1997,
LangSchaudtStein2017}. For sufficiently large complete graphs,
S\'ark\"ozy~\cite{Sarkozy2020} subsequently improved the dependence for
locally coloured complete graphs to $O(r\log r)$.

Relaxing the condition on the colouring even more, one could allow a small
exceptional set of vertices to have large palettes, which is not the
case in $r$-local colourings. This motivates the definition of an
\emph{$r$-mean} colouring of an $n$-vertex graph $G$, which is a colouring satisfying 
\[
  \frac{\sum_{v\in V(G)}\dd(v)}{n} \le r\,.
\]
Mean colourings arose in the Ramsey-theoretic work of Caro \cite{Caro1992} and
were subsequently studied systematically by Caro and Tuza \cite{CaroTuza1993}.
Conlon and Stein also initiated the monochromatic cycle-partition problem for
mean colourings in complete graphs. They proved that every $2$-mean coloured
complete graph can be partitioned into two monochromatic cycles of different
colours, and asked whether, for every fixed $r$, there is a constant $C$ such
that every $r$-mean coloured complete graph has a partition into at most $C$
monochromatic cycles. We settle the first open case by proving the analogue of
Conlon and Stein's result for $r=3$. 

\begin{theorem}\label{thm:main}
  There is an absolute constant $C$ such that the vertex set of every
  $3$-mean edge-coloured complete graph can be partitioned into at most $C$
  disjoint monochromatic cycles.
\end{theorem}

\section{Preliminaries}\label{sec:preliminaries}

We use three known results at different stages of the proof. An old result of P\'osa's (\cref{lem:posa})
converts a bound on the independence number of an auxiliary graph into a cycle
partition; we shall use it to prove our one-sided bipartite covering lemma.
Conlon and Stein's work on local colourings (\cref{thm:local-complete}) serves as a cleaning tool: it covers both the vertices of intermediate colour-degree and the final residue of
constant colour-degree.
Finally, the balanced bipartite theorem of Benevides, Quintino, and Talon (\cref{Lemma_2_coloured}) is used in two instances to cover a $2$-edge-coloured remainder.

\begin{theorem}[P\'osa \cite{Posa1963}]\label{lem:posa}
  For any graph $G$ there is a partition of $V(G)$ into at most
  $\alpha(G)$ disjoint cycles.
\end{theorem}

We use the next result as our complete graph cleaning tool. Notice that passing to
an induced subgraph cannot increase any colour-degree, so a class defined by
an upper bound on its colour-degrees automatically satisfies the required local
hypothesis.

\begin{theorem}[Conlon--Stein \cite{ConlonStein2016}]
  \label{thm:local-complete}
  For every positive integer $r$, there is a constant $g(r)$ such that the
  vertex set of every $r$-locally coloured complete graph can be partitioned
  into at most~$g(r)$ disjoint monochromatic cycles.
\end{theorem}

We remark that the proof of~\cref{thm:local-complete} gives $g(r)=15r^2\log r$. To simplify the numbers, we use this estimate in our proof. We also need the following result.

\begin{theorem}[Benevides--Quintino--Talon~\cite{BenevidesQuintinoTalon2025}]\label{Lemma_2_coloured}
Let $G$ be a $2$-edge-coloured balanced complete bipartite graph. Then $G$ can be covered by $4$ vertex disjoint monochromatic cycles.
\end{theorem}

The next lemma is the interface between the dense monochromatic part of the
argument and P\'osa's lemma.  It is deliberately one-sided: only the smaller
part must be covered.
Its proof is adapted from the work of Erd\H{o}s, Gy\'arf\'as and Pyber~\cite{ErdosGyarfasPyber1991}.

\begin{lemma}\label{Lemma_uncoloured}
If $G$ is a bipartite graph with parts $X$ and $Y$ such that $|X|\geq 1000|Y|$ and~$\delta(Y)\geq |X|/10$. Then $Y$ can be covered by at most~$10$ vertex disjoint cycles, which cover at most $|Y|$ vertices in $X$.
\end{lemma}

\begin{proof}
  Assume $|Y| \geq 11$ as otherwise the result follows trivially. Put $y=|Y|$ and
  let $N_G(u,v)$ denote the common neighbourhood of vertices $u$ and $v$ in $G$. In
  order to apply~\Cref{lem:posa}, define an auxiliary graph $J$ on $Y$ by
  putting an edge $uv$ in $J$ if and only if $|N_G(u,v)|\ge y$.
  We claim $\alpha(J)\le 10$. In fact, take distinct vertices
  $v_1,\dots, v_{11}\in Y$. For $w\in X$, let $t_w$ be the number of these 
  vertices adjacent to $w$. Since $\binom{t}{2}\ge t-1$ for every
  integer $t \geq 0$, we have
  \[
    \sum_{1\le i<j\le 11}|N_G(v_i,v_j)|
      =\sum_{w\in X}\binom{t_w}{2}
      \ge\sum_{w\in X}(t_w-1)
      \ge\frac{11|X|}{10}-|X|
      = \frac{|X|}{10}\ge100|Y|\,.
  \]
  Thus some pair has at least $y$ common neighbours and is an edge of $J$. 
  Therefore, any 11 vertices of $J$ span an edge, proving the claim that $\alpha(J)\le 10$.

  By \cref{lem:posa}, $V(J)$ can be partitioned into at most $10$
  cycles, which we now realise as cycles in $G$. For every pair
  $u,v$ of consecutive vertices on a cycle of order at least three in $J$,
  request a vertex $z_{uv}$ from $N_G(u,v)$. If a cycle of $J$ has order
  two, and hence consists of a single edge $uv$, request distinct vertices
  $z_1,z_2$ from $N_G(u,v)$. No vertex is requested for a singleton cycle. The
  total number of requests is the sum of the orders of the non-singleton
  cycles, which is at most $y$. From the definition of $J$, each requested common neighbourhood has size at least $y$ and we may therefore choose
  all the requested vertices to be distinct.
  For a cycle of order at least three in $J$, replace each of its
  edges $uv$ by the path~$u z_{uv}v$, giving a cycle in $G$. For a
  cycle of order two in $J$, the four-vertex cycle~$uz_1vz_2u$ is a cycle of
  $G$. Finally, retain every singleton cycle as a singleton in $G$. Since 
  the chosen vertices are distinct, this yields at most $10$ vertex-disjoint
  cycles of $G$ covering~$Y$, using at most $y$ vertices of $X$.
\end{proof}

\section{Proof of \cref{thm:main}}

Given a colouring of $E(K_n)$ such that the average number of colours on edges incident to a vertex is at most $3$, we will show that the set of vertices in $V=V(K_n)$ can be covered by at most $5\cdot 10^8$ vertex disjoint monochromatic cycles.

\begin{proof} By the colouring assumption we have
$\sum_{v\in V} d_{\chi}(v) \leq 3 n$.
Partition the vertices in $V$ into three sets $A$, $B$ and $C$ defined as follows:
\begin{equation*}
A=\{v\colon d_\chi(v)\le 2\}\,, \quad
B=\{v\colon d_\chi(v)\ge 2003\}\quad\text{and}\quad
C=\{v\colon 3 \le d_\chi(v)\le 2002\}\,.
\end{equation*}
Note that since 
$$3|A|+3|B|+3|C|=3n\geq \sum_{v\in V} d_{\chi}(v) \geq |A|+2003|B|+3|C|,$$ we obtain $2|A|\geq 2000|B|$ or $|A|\geq 1000|B|$.

Let $P$ be a monochromatic star of order $|A|/2$ in $A$, and let $c_0$ be the colour of $P$. Define 
\begin{equation*}
B_0:=\{b\in B\colon |N_{c_0}(b)\cap A|\geq |A|/10\} \quad \text{and} \quad
B_1:= B\setminus B_0\,.
\end{equation*}
Then, we can apply Lemma \ref{Lemma_uncoloured} to the $c_0$-coloured bipartite graph between $B_0$ and $A$ to conclude that we can cover $B_0$ by at most $10$ vertex disjoint $c_0$-coloured cycles.
Denote the set of remaining vertices in $P$ by $P'$. Note that
$$|P'| \geq |P| - |B_0| \geq |A| \left( 1/2 - 1/1000 \right)=499|A|/1000\,.$$

Recall that all vertices in $P'\subseteq A$ touch at most two colours, one of which is $c_0$. 
Let~$c_i$ be the $i$-th most popular colour other than $c_0$ seen by vertices of $P'$, and $D_i=\{v\in P'\colon \text{the colours on $v$ are $c_0, c_i$}\}$.
Then, $|D_i|\leq \ldots \leq |D_1|$. If $|D_1|\geq |A|/1000$, then we can cover $B_1$ by 4 vertex disjoint monochromatic cycles in the bipartite graph between $B_1$ and $D_1$
by applying \cref{Lemma_2_coloured} (in this case $|B_1| \leq |A|/1000 \leq |D_1|$, so just keep $|B_1|$ vertices from $D_1$ to get a balanced 2-coloured bipartite complete graph). Thus we may assume 
$$|D_i|\leq \ldots \leq |D_1|\leq |A|/1000\,.$$
Our goal will be to show $|B_1| \leq \sum_{i=1}^6 |D_i|$. For this purpose, first we show the following.

\begin{claim}
Every vertex $b\in B_1$ sees at least $394|A|/(1000|D_6|)$ colours going into $P'$.
\end{claim}
\begin{proof}
Note that by the definition of $b\in B_1$, the vertex $b$ touches fewer than $|A|/10$ edges with colour $c_0$ that go to $P'$.
Furthermore, $b$ touches at most $|D_i|$ edges with colour $c_i$ for $1\leq i \leq 5$ and at most $|D_6|$ edges with colour $c_i$ for $6\leq i$ that go into $P'$.
Otherwise we would get a contradiction with the definition of the colours~$c_i$. 
Let the number of colours from $b$ to $P'$ equal $s$. We have established 
\[
\frac{499|A|}{1000}
\leq |P'|
=|N_{P'}(b)|
\le \frac{|A|}{10}+\sum_{i=1}^5|D_i| +s|D_6|
\leq  \frac{|A|}{10}+\frac{5|A|}{1000}+s|D_6|\,,\]
which rearranges to give $s\ge 394|A|/(1000|D_6|)$.
\end{proof}

Note that
$$3n\ge \sum_{v\in V} d_{\chi}(v) \ge |A|+|B_1|\max\left(2003, \frac{394|A|}{1000|D_6|}\right)+ 3(n-|A|-|B_1|)\,.$$
Keep only one term of the max, and rearrange this into
\[
2|A|
\geq |B_1|\left(\frac{394|A|}{1000|D_6|}-3\right)
\geq |B_1|\cdot\frac{391|A|}{1000|D_6|}\,.\]
Rearrange again to get
\[
|B_1|
\le \frac{2000|D_6|}{391}
\le 6|D_6|
\leq \sum_{i=1}^6 |D_i|\,.
\]
We partition $B_1 = \bigcup_{i=1}^6 B_1^i$, such that $|B_1^i|\leq |D_i|$.
From \cref{Lemma_2_coloured}, we may cover $B_1^i$ by 4 monochromatic cycles in the 2-coloured bipartite graph between $B_1^i$ and $D_i$ (just keep $|B_1^i|$ vertices from $D_i$ to get a balanced 2-coloured bipartite complete graph). This gives 24 vertex disjoint monochromatic cycles and now we have covered
all vertices in $B$. Denote the set of remaining vertices in $A$ by $A'$. \cref{thm:local-complete} ensures that at most $15\cdot 2002^2 \cdot \log{(2002)}$ disjoint monochromatic cycles cover $A'\cup C$, which gives a total of at most $5\cdot 10^8$ cycles.
\end{proof}

\section*{Acknowledgments}

The proof of \cref{thm:main} was found independently by Lang--Mota and Bir\'o--Martin--Pokrovskiy--Ruszink\'o--S\'ark\"ozy using related approaches. After learning from each others' result, we decided to merge the work. Part of the second group of authors' research was done
at a conference celebrating the 80th birthdays of Andr\'as Gy\'arf\'as and Jen\H{o} Lehel hosted by the Erd\H{o}s Center at the R\'enyi Institute. These authors would like to thank Maria Axenovich, J\'anos Bar\'at, Louis DeBiasio, and Andr\'as Gy\'arf\'as for discussions related to this project. 

R. Lang was supported by the Ram\'on y Cajal programme (RYC2022-038372-I) and by
grant PID2023-147202NB-I00 funded by MICIU/AEI/10.13039/501100011033. G. O. Mota
was supported by CNPq (420838/2025-2 and 315916/2023-0) and FAPESP (2024/13859-4
and 2023/03167-5). G. S\'ark\"ozy was supported by NKFIH Grants No. K132696, K117879.

\bibliographystyle{plain}
\bibliography{main}

\end{document}